\documentclass[a4paper]{amsart}

\usepackage{amsmath,amssymb}
\usepackage{amsthm}
\usepackage{enumitem}
\usepackage{mathtools}
\usepackage[all]{xy}
\usepackage{bbm}

\theoremstyle{plain}
\newtheorem{theorem}{Theorem}[section]
\newtheorem{proposition}[theorem]{Proposition}
\newtheorem{lemma}[theorem]{Lemma}
\newtheorem{corollary}[theorem]{Corollary}

\theoremstyle{definition}
\newtheorem{definition}[theorem]{Definition}
\newtheorem*{acknowledgements}{Acknowledgements}
\newtheorem{setting}[theorem]{Setting}

\theoremstyle{remark}
\newtheorem{remark}[theorem]{Remark}
\newtheorem{notation}[theorem]{Notation}
\newtheorem{convention}[theorem]{Convention}

\numberwithin{equation}{theorem}

\DeclareMathOperator{\Pic}{Pic}
\DeclareMathOperator{\NE}{NE}

\DeclareMathOperator{\pr}{pr}

\DeclareMathOperator{\Nef}{Nef}

\DeclareMathOperator{\Gr}{Gr}
\DeclareMathOperator{\rank}{rank}
\DeclareMathOperator{\RC}{RatCurves}
\DeclareMathOperator{\Ext}{Ext}
\DeclareMathOperator{\Exc}{Exc}
\DeclareMathOperator{\codim}{codim}

\DeclareMathOperator{\td}{td}
\DeclareMathOperator{\ch}{ch}
\DeclareMathOperator{\pt}{pt}

\newcommand{\cNE}{{\overline{\NE}}}

\newcommand{\nequiv}{\equiv _\mathrm{num}}
\newcommand{\tM}{\widetilde M}
\newcommand{\tN}{\widetilde N}
\newcommand{\tf}{\widetilde f}
\newcommand{\te}{\widetilde e}
\newcommand{\tR}{\widetilde R}

\newcommand{\blank}{{-}}

\newcommand{\tvarphi}{\widetilde \varphi}

\newcommand{\Banica}{B\u anic\u a }

\newcommand{\sE}{\mathcal{E}}
\newcommand{\sF}{\mathcal{F}}
\newcommand{\sG}{\mathcal{G}}
\newcommand{\sL}{\mathcal{L}}

\newcommand{\sC}{\mathcal{C}}
\newcommand{\sS}{\mathcal{S}}
\newcommand{\sN}{\mathcal{N}}

\newcommand{\cO}{\mathcal{O}}

\newcommand{\bC}{\mathbf{C}}
\newcommand{\bF}{\mathbf{F}}

\newcommand{\bP}{\mathbf{P}}
\newcommand{\bQ}{\mathbf{Q}}
\newcommand{\bR}{\mathbf{R}}
\newcommand{\bZ}{\mathbf{Z}}

\newcommand{\tC}{\widetilde{C}}
\newcommand{\tW}{\widetilde{W}}
\newcommand{\tY}{\widetilde{Y}}
\newcommand{\tE}{\widetilde{E}}
\newcommand{\tB}{\widetilde{B}}

\AtBeginDocument{%
   \def\MR#1{}
}

\setenumerate{label=(\arabic*),
font=\normalfont
}

\title[Classification of Mukai pairs with dimension $4$ and rank $2$]{Classification of Mukai pairs with \\dimension $4$ and rank $2$}
\author[A. KANEMITSU]{Akihiro KANEMITSU}
\date{\today}
\address{Research Institute for Mathematical Sciences,
Kyoto University, Kyoto 606-8502, Japan}
\email{kanemitu@kurims.kyoto-u.ac.jp}
\thanks{The author is a JSPS Research Fellow and he was supported by the Grant-in-Aid for JSPS fellows (JSPS KAKENHI Grant Number 15J07608,  18J00681).
This work was supported by the Program for Leading Graduate Schools, MEXT, Japan.}

\begin{document}

\begin{abstract}
We give the complete classification of Mukai pairs of dimension $4$ and rank $2$ with Picard number one, that is, pairs $(X,\sE)$ where $X$ is a Fano $4$-fold with Picard number one, and $\sE$ is an ample vector bundle of rank two on $X$ with $c_1(X) = c_1(\sE)$.
Equivalently, the present paper completes the classification of ruled Fano $5$-folds with index two, which was partially done by C.~Novelli and G.~Occhetta in 2007.
\end{abstract}

\maketitle

\section*{Introduction}
\subsection{}
A \emph{Fano manifold} $M$ is, by definition, a smooth projective variety with ample anticanonical divisor $-K_M$.
One of the most important and fundamental invariants for a given Fano manifold $M$ is its \emph{index}, denoted by $r_M$, which is defined as the greatest integer dividing $-K_M$ in the Picard group.
Roughly speaking, the index measures how large the anticanonical divisor is, and philosophically the positivity of anticanonical divisor poses some restriction on the (biregular) structure of $M$.
Indeed, Fano manifolds $M$ with $r_M \geq \dim M -2$ are completely classified in celebrated articles \cite{KO73,Fuj82a,Fuj82b,Muk89} (cf.\ \cite{Mel99,Amb99}).

On the other hand, it appears that, due to lack of knowledge on Calabi-Yau $3$-folds or Fano $4$-folds, the complete classification of Fano $n$-folds with index $n-3$ is far from being complete, particularly when the Picard number is one.
Nevertheless, if $n \geq 5$ and the Picard number is bigger than 1, then we still have some room to attack the problem by studying its contractions, whose existence is promised by the fundamental theorems in Mori theory.
It is known that, if $M$ is a Fano $n$-fold with index $n-3$ and $\rho_M>1$,  then $n\leq 8$  \cite{Wis90}.
Moreover, by works on Fano manifolds with large index  \cite{Wis90,Wis91,PSW92b,Wis93,BW96} (cf.\ \cite{Occ05}), such Fano manifolds are completely classified when $n \geq6$.
Also, in a series of papers \cite{CO06,NO07,CO08}, Chierici, Novelli and Occhetta started to classify Fano $5$-folds with index two and Picard number bigger than one by the above strategy, though the classification is not completed yet.

Among their study, the ruled case \cite{NO07} is of particular interest in the present paper.
In that paper, Novelli and Occhetta classified ruled Fano $5$-folds $M$ with index two with the assumption $\rho_M \geq 3$ (see Remark~\ref{rem:NO} for their treatment in the case $\rho_M=2$).
The purpose of this paper is to complete the classification of ruled Fano $5$-folds with index two by studying the case $\rho_M =2$, which are not treated sufficiently in \cite{NO07}.

\subsection{}
There is another point of view; \emph{classification of Mukai pairs with large rank}.
In 1988, Mukai \cite{Muk88} introduced study of pairs $(X,\sE)$ where $X$ is a Fano manifold, and $\sE$ is an ample vector bundle on $X$ with $c_1(X) = c_1(\sE)$.
In \cite{Kan17}, such a pair is called a \emph{Mukai pair}.
Given a Mukai pair $(X,\sE)$, its rank is defined as the rank of the bundle $\sE$.
This invariant ``rank $\sE$ for a Mukai pair'' is an analogue of ``index for a Fano manifold.''
For example, a Fano manifold $X$ with index $r_X$ gives a Mukai pair $(X, \cO(H_X)^{\oplus r_X})$ of rank $r_X$, where $H_X\coloneqq -K_X/r_X$ is the fundamental divisor of $X$.
Based on the above analogy, Mukai pairs are classified when $\rank \sE \geq \dim X -1$ around 1990s \cite{Fuj92,Pet90,Pet91,YZ90,Wis89b,PSW92b} (cf.\ \cite{Occ05}).

Therefore, it is natural to hope the classification of Mukai pairs with $\rank \sE = \dim X -2$ as a next step.
In \cite{Kan17}, the author classified such pairs when $\dim X \geq 5$.
Note that,  for the case $\rank \sE = \dim X -2$, the smallest possible value of $\dim X$ is three, and the classification in this case is equivalent to the classification of Fano $3$-folds, which is established by Fano, Iskovskih, Shokurov, Fujita, Mori and Mukai (see \cite{IP99} and references therein).
On the other hand, the classification of Mukai pairs with dimension $4$ and rank $2$ is equivalent to the classification of ruled Fano $5$-folds with index two by taking the projectivization $\bP(\sE)$, and, as we mentioned, such a classification is partially done by Novelli and Occhetta.
Thus, when we are dealing with the classification problem of Mukai pairs with $\rank \sE = \dim X-2$, the remaining part to be considered is the case $\dim X=4$, $\rank \sE =2$ and $\rho _X =1$.
The purpose of this paper is to classify Mukai pairs in this missing case, and hence to complete the classification of Mukai pairs of  $\rank \sE = \dim X -2$ in arbitrary dimension. 

\begin{theorem}[Classification of Mukai pairs with dimension $4$ and rank $2$]\label{thm:main}
Let $(X,\sE)$ be a Mukai pair of dimension $4$ and rank $2$.
Assume that $\sE$ is indecomposable and $\rho _X =1$.
Then $(X,\sE)$ is isomorphic to either% one of the following pairs:
\[
(V_4, p^*\sS^* _{\bQ^4} (1)) \text{ or } (V_5, \sS ^*_{V_5} (1)).
\]
Here we use the following symbols:
\begin{itemize}
 \item $V_4$ is a quartic del Pezzo $4$-fold obtained as a double cover $p \colon V_4 \to \bQ ^4$ of the hyperquadric of dimension $4$ branched along a smooth divisor $B \in \left| \cO_{\bQ^4}(2) \right|$.
 \item $\sS_{\bQ^4}$ is a spinor bundle on $\bQ^4$.
 \item $V_5$ is a general linear section of the Grassmannian $\Gr (2,5)$ embedded into the projective space $\bP ^9$ via the Pl\"ucker embedding.
 \item $\sS _{V_5}$ is the restriction of the universal subbundle on $\Gr(2,5)$ to $V_5$.
\end{itemize}
\end{theorem}

\begin{remark}[Novelli and Occhetta's result]\label{rem:NO}
The classification problem of pairs as in the assumption of Theorem~\ref{thm:main} was also discussed in Section~8 of Novelli and Occhetta's paper \cite{NO07}.
Essentially, they classified such pairs with the extra assumption $r_X \geq 2$, which is not required in this paper.
Also, since the author could not follow their argument (see Remark~\ref{rem:comp}), we will provide a proof of Theorem~\ref{thm:main} that is independent from their argument.
\end{remark}

\subsection{}
Here we briefly sketch the strategy of the proof and give an outline of this paper.
Let $(X,\sE)$ be a Mukai pair as in Theorem~\ref{thm:main}.
  Then $W \coloneqq \bP(\sE)$ is a Fano $5$-fold with index two and $\rho_W=2$.
 We will denote by $\pi \colon W \to X$ the natural projection.
 Since $\rho _W =2$, it admits another contraction $\varphi \colon W \to Y$.
Note that, in each outcome of Theorem~\ref{thm:main}, 
\begin{enumerate}
\item the corresponding manifold $Y$ is isomorphic to a projective space $\bP^m$ and 
\item each $\pi$-fiber gives a line in $Y \simeq \bP^m$.
\end{enumerate}
In our proof, a key step consists of proving these properties for our Mukai pair $(X,\sE)$.
Then we have a morphism $X \to \Gr (2,m+1)$, which will be proved to satisfy the desired description of the pair $(X,\sE)$.

In Section~\ref{sect:pre}, we gather some results on fundamental extremal contractions, such as scrolls, quadric bundles and blow-ups.
Those results will be applied to $\varphi$ in later sections.

In Section~\ref{sect:RC}, we will study rational curves on a Mukai pair $(X,\sE)$ as in Theorem~\ref{thm:main}.
There we will see that $X$ is covered by rational curves with anticanonical degree $3$.
Such rational curves are called \emph{minimal rational curves}.
For a minimal rational curve $f \colon \bP^1 \to X$, the pull-back $f^*\sE$ is isomorphic to $\cO(2) \oplus \cO (1)$, and thus the pull-back of $\bP(\sE)$ over $\bP^1$ is the Hirzebruch surface $\bF_1$.
Then the minimal section of this ruled surface gives a rational curve on $\bP(\sE)$, which we call the \emph{minimal lift} of the minimal rational curve $f \colon \bP^1 \to X$.

In Section~\ref{sect:stable}, we will see that the bundle $\sE$ is stable, and hence the Bogomolov inequality holds for $\sE$ (Proposition~\ref{prop:Bogomolov}).
This inequality is crucial in our proof of Theorem~\ref{thm:main}, and will be used twice: first to give rough description of the elementary contraction $\varphi$ of $\bP(\sE)$ (Section~\ref{sect:rough}), and second to prove that minimal lifts of minimal rational curves are contracted by $\varphi$ (Propositions~\ref{prop:Ineq} and \ref{prop:comp}). 

In Section~\ref{sect:rough}, we will give a rough description of $\varphi$; the morphism $\varphi$ is  either a quadric bundle or a special \Banica scroll (see Definition~\ref{def:sBscroll} and Theorem~\ref{thm:phi}). 
Then, in Sections~\ref{sect:quarticDP} and \ref{sect:quinticDP}, we will complete the proof of Theorem~\ref{thm:main} by studying each case.

\medskip
In the subsequent sections, we frequently use the following setting and notations.
\begin{setting}[Mukai pairs of dimension $4$ and rank $2$]\label{set}
$(X,\sE)$ is a Mukai pair of dimension $4$ and rank $2$ such that $\sE$ is indecomposable and $\rho _X =1$.

 We will use the following notations:
\begin{itemize}
 \item $W \coloneqq \bP(\sE)$ is the Grothendieck projectivization of $\sE$.
 \item $\pi \colon W \to X$ is the natural projection.
 \item $\varphi \colon W \to Y$ is the other elementary contraction.
 \item $\xi=\xi_\sE$ is the tautological divisor of the projectivization $\bP(\sE)$.
 \item $H_X$ (resp.\ $H_Y$) is the ample generator of $\Pic(X)$ (resp.\ $\Pic(Y)$).
 \item $r_X$ is the index of $X$.
 \item $l_X$ is the length of $X$.

 \item $R_\pi \subset \NE(X)$ (resp.\ $R_\varphi\subset \NE(X)$) the ray corresponding to $\pi$ (resp.\ $\varphi$).
 \item $\tR \subset \NE(X)$ is the half line $\bR_{\geq0}[\tC]$ spanned by the class $[\tC]$ of minimal lifts of minimal rational curves.
\end{itemize}
\end{setting}

\begin{convention}
 \hfill
\begin{enumerate}
 \item For a morphism $f \colon X \to Y$ and a coherent sheaf $\sF$ on $Y$, we will denote by $\sF|_X$ the pull-back $f^*\sF$ if no confusion arises.
 \item A \emph{rational curve} $C$ on a Fano manifold $X$ is, by definition, a $1$-dimensional closed subvariety whose normalization is a projective line $\bP^1$.
Sometimes the normalization map $f \colon \bP^1 \to X$ itself is also referred to as a rational curve.
\end{enumerate}
\end{convention}

\begin{acknowledgements}
The author wishes to express his gratitude to Professors Shigeru Mukai and Hiromichi Takagi for their valuable comments and suggestions.
He is also grateful to Professor Gianluca Occhetta for his helpful comments on the paper \cite{NO07}.
The author is also grateful to Doctor Takeru Fukuoka for his helpful comments and discussions on quintic del Pezzo manifolds.
He also wishes to thank Professor Hiraku Kawanoue for useful comments on the proof of Proposition~\ref{prop:Ineq}. 

Main part of this work was done at the Graduate School of Mathematical Sciences, the University of Tokyo, and this paper was prepared at Research Institute for Mathematical Sciences,
Kyoto University.
\end{acknowledgements}

\section{Preliminaries: fundamental extremal contractions}\label{sect:pre}

Here we briefly review some facts on fundamental extremal contractions.
Let $M$ be a smooth projective variety and $\varphi \colon M \to N$ be its elementary contraction, that is, a contraction associated with a $K_M$-negative extremal ray $R$ of $\cNE(M)$.
Recall that the length $l(R_\varphi)$ of the ray $R$ is defined as the minimum anticanonical degree of rational curves that are contracted by $\varphi$:
\[
l(R)\coloneqq\{\, -K_M\cdot C \mid\, \text{$C$ is a rational curve such that $[C] \in R$} \}.
\]
The following is a fundamental inequality, which gives a bound of dimensions of the exceptional locus and fibers from below in terms of length.
\begin{theorem}[{Ionescu-Wi\'sniewski inequality \cite[Theorem~0.4]{Ion86}, \cite[Theorem~1.1]{Wis91}}]\label{thm:IWineq}
 Let $\varphi \colon M \to N$ be an elementary contraction associated to an extremal ray $R$, $E$ an irreducible component of $\Exc (\varphi)$ and $F$ a fiber contained in $E$.
 Then
 \[
 \dim E + \dim F \geq \dim M + l(R) -1.
 \]
\end{theorem}

Assume that the ray $R$ is supported by a Cartier divisor of the form $K_M+rL$, where $L$ is a $\varphi$-ample Cartier divisor and $r \in \bZ_{>0}$ is a positive integer.
Then $l(R) \geq r$, and thus we have
\[
 \dim E + \dim F \geq \dim M + r -1.
\]
Therefore, if $\varphi$ is of fiber type, then $\dim F \geq r-1$.
Hence, by taking $F$ as a general $\varphi$-fiber, we have
\[
\dim M-\dim N +1\geq r.
\]

\begin{definition}[Scrolls and quadric fibrations]
An elementary contraction $\varphi$ is called an \emph{adjunction theoretic scroll}, or simply a \emph{scroll}, if it is of fiber type and there is a $\varphi$-ample line bundle $L \in \Pic (M)$ such that 
\[
K_M + (\dim M -\dim N +1)L
\]
is a supporting divisor for the contraction $\varphi$.

Similarly, the contraction is called a \emph{quadric fibration} if it is of fiber type and there is a $\varphi$-ample line bundle $L \in \Pic (M)$ such that 
\[
K_M + (\dim M -\dim N)L
\]
is a supporting divisor for the contraction $\varphi$.
\end{definition}
Note that in the above definition $\varphi$ is assumed to be  elementary.
Also, in each case, a general $\varphi$-fiber is isomorphic to $\bP^n$ or $\bQ^n$ respectively. 

The projection of a projectivized vector bundle $\bP(\sF) \to N$ is a typical example of scrolls.
Conversely, Fujita proved that an equidimensional scroll $\varphi \colon M \to N$ is a projection of a projectivized vector bundle \cite[Lemma~2.12]{Fuj87}.
The corresponding result for quadric fibrations is proved by Andreatta, Ballico and Wi\'sniewski:
\begin{theorem}[Characterization of quadric bundles {\cite[Theorem~B]{ABW93}}]\label{thm:Qbundle}
 Let $\varphi \colon M \to N$ be an equidimensional quadric fibration (from a smooth manifold $M$).
 Then $N$ is smooth and $\varphi$ is a quadric bundle, that is,  there exists a vector bundle $\sF$ of rank $\dim M -\dim N +2$ on $N$ such that $M$ is embedded into $\bP(\sF)$ as a divisor of relative degree two.
\end{theorem}

We also need the following characterization of bolw-ups.
\begin{theorem}[Characterization of blow-ups {\cite[Theorem~4.1]{AW93}}]\label{thm:Bl}
Let $r\in \bZ_{>0}$ be a positive integer and $\varphi \colon M \to N$ an elementary birational contraction whose nontrivial $\varphi$-fibers have dimension $r$.
Assume that there is a $\varphi$-ample line bundle $L$ such that  $K_M + rL$ gives a supporting divisor of the contraction $\varphi$.
Then $N$ is smooth and $\varphi$ is a blow-up of N along a smooth subvariety of codimension $r+1$. 
\end{theorem}

\subsection{Special \Banica scrolls}
A special kind of scrolls, called \emph{\Banica scrolls}, are introduced in \cite{BW96}.
Here we will briefly recall the definition  and some properties of  \Banica sheaves and \Banica scrolls, based on \cite{BW96,AW93}. 

\begin{definition}[\Banica sheaf]
 A \emph{\Banica sheaf} $\sF$ on a normal variety $N$ is a coherent sheaf $\sF$ on $N$ whose projectivization $\bP(\sF)$ is a smooth variety.
\end{definition}
Let $\sF$ be a \Banica sheaf on a normal variety $N$.
Then $M\coloneqq \bP(\sF)$ is a smooth variety.
Let $\varphi \colon M \to N$ be its projection.
Then it is known that the sheaf $\sF$ is reflexive and that the contraction $\varphi$ is an elementary contraction \cite[Lemma~2.2]{BW96}.
Moreover, the contraction $\varphi$ is a scroll.
A scroll obtained in this way is called a \emph{\Banica scroll}.
Conversely, a sufficient condition for a scroll to be a \Banica scroll is given in \cite{AW93,BW96} as follows:

\begin{theorem}[{\cite[Theorem~4.1, Remark~4.12]{AW93}, cf.\ \cite[Proposition~2.5]{BW96}}]\label{thm:Bscroll}
 Let $\varphi \colon M \to N$ be a scroll and  $L \in \Pic (M)$ a $\varphi$-ample line bundle on $M$ that gives a supporting divisor
\[
K_M + (\dim M -\dim N +1)L
\]
for the contraction $\varphi$.

Assume that
\[
\dim F \leq \dim M -\dim N +1
\]
holds for every $\varphi$-fiber $F$.
Then $N$ is smooth and $M \simeq \bP(\varphi_*L)$.
In particular, $\varphi_*L$ is a \Banica sheaf.
\end{theorem}

\begin{definition}[Special \Banica scrolls]\label{def:sBscroll}
 A scroll satisfying the assumptions in Theorem~\ref{thm:Bscroll} is called a \emph{special \Banica scroll}. 
\end{definition}

A more geometric description of special \Banica scrolls is presented in \cite{AW93} as follows:
Let $\varphi \colon M \to N$ be a special \Banica scroll and $Z \subset N$ the set of points over which the dimension of fibers is $\dim M - \dim N +1$.
Set $E \coloneqq \varphi^{-1} (Z)$.

\begin{proposition}[{Geometry of special \Banica scrolls \cite[Remark~4.13]{AW93}}]\label{prop:Bscroll}
 Let the notation be as above.
 Then the following hold:
\begin{enumerate}
 \item $Z$ is smooth of codimension $\dim M -\dim N +2$.
 \item The morphism $\varphi|_{E} \colon E \to Z$ is a smooth $\bP^{\dim M -\dim N +1}$-bundle.
 \item We have the following commutative diagram with $\bP^{\dim M -\dim N}$-bundle $\tvarphi$
 \[
 \xymatrix{
 \tM \ar[r]^-{\tvarphi} \ar[d] & \tN \ar[d]\\
 M \ar[r]^-\varphi & N,
 }
 \]
 where $\tM \to M$ is the blow-up of $M$ along $E$, and $\tN \to N$ the blow-up of N along $Z$.
\end{enumerate}
\end{proposition}

\section{Rational curves and Mukai pairs}\label{sect:RC}

 Let $X$ be a smooth Fano manifold.
Recall that a rational curve $f \colon \bP^1 \to X$ is called \emph{free} if the pull-back of the tangent bundle $f^*T_{X}$ is a nef vector bundle on $\bP^1$.
Then the \emph{length} $l_X$ is defined as the minimum anticanonical degree of free rational curves on $X$.  
Note that we have $l_X \geq r_X \geq 1$ by definition and that these values are at most $\dim X +1$ by Mori's theorem \cite{Mor79}.
This value $l_X$ is of particular interest because of the following  theorem, which gives a characterization of projective space and hyperquadric in terms of length.
\begin{theorem}[{Characterization of projective space and hyperquadric \cite{CMSB02,Miy04a}, cf.\ \cite{Keb02,DH17}}]\label{thm:PQ}
Let $X$ be a Fano manifold with Picard number one.
If $l_X \geq \dim X$, then $X$ is isomorphic to either a projective space $\bP^n$ or a hyperquadric $\bQ^n$.
\end{theorem}

In this section, we will prove $l_X=3$ for a Mukai pair as in Setting~\ref{set} (Lemma~\ref{lem:lX}).
Thus, we have a family of rational curves on $X$ whose anticanonical degree is $3$.
Further, we will construct minimal liftings of these rational curves (Definition~\ref{def:lift} and Lemma~\ref{lem:lift}).

The following lemma will be frequently used later.
\begin{lemma}[Splitting type of $\sE$]\label{lem:splitting}
 Let $(X,\sE)$ be a Mukai pair of dimension $n$ and rank $r$, and $f \colon \bP^1 \to C \subset X$ a rational curve on $X$.
 Then the following hold:
\begin{enumerate}
 \item \label{lem:splitting1} $f^* \sE \simeq \bigoplus_{i=1}^r \cO(a_i)$ for positive integers $a_i \in \bZ_{>0}$ such that $\sum_{i=1}^r a_i = -K_X \cdot C$.
 \item \label{lem:splitting2} $n+1 \geq l_X \geq r$.
\end{enumerate}
\end{lemma}
\begin{proof}
By Grothendieck's theorem, every vector bundle on $\bP^1$ is isomorphic to a direct sum of line bundles.
Thus, $f^* \sE \simeq \bigoplus_{i=1}^r \cO(a_i)$ for integers $a_i \in \bZ$.
Since $\sE$ is ample and $c_1(\sE)=c_1(X)$, each $a_i$ is positive and $\sum_{i=1}^r a_i = -K_X \cdot C$.
This proves the first assertion.

As we mentioned, $l_X \leq \dim X +1$ by Mori's theorem \cite{Mor79}.
Also, it follows from \ref{lem:splitting1} that $l_X \geq r$.
\end{proof}

 A \emph{family of rational curves} on $X$ is an irreducible component $M$ of the parameter space $\RC ^n (X)$  for the rational curves on $X$.
For an account of the theory of rational curves, we refer the reader to \cite{Kol96}.
Given a family $M$ of rational curves, there is the following diagram, which gives the realization of the family:
\begin{equation}\label{dia:UF}
\xymatrix{
 U \ar[r]^-e \ar[d]^-p & X \\
 M.
}
\end{equation}
Here $p \colon U \to M$ is the universal family of rational curves, which is known to be a smooth $\bP^1$-fibration,  and $e$ is the evaluation map of rational curves.
Roughly speaking, each point $m \in M$ corresponds to a rational curve $C \subset X$, and the map $e|_{p^{-1}(m)} \colon p^{-1}(m) \to X$ gives the normalization map of $C$.
Given a family $M$ of rational curves, its \emph{anticanonical degree} is defined as the anticanonical degree of a rational curve parametrized by $M$.

Recall that a family $M$ of rational curves is called \emph{covering} (resp.\ \emph{dominating}) if the evaluation map $e$ is surjective (resp.\ \ \emph{dominating}), and it is called \emph{unsplitting} if $M$ is proper.
Note that a family of rational curves is dominating if and only if a general rational curve in this family is free \cite[Proposition~1.1]{KMM92a}.

The next proposition is formulated in \cite{Kan17} for $n\geq 5$, while its proof works for a slightly weaker assumption $n\geq4$.
For the proof, we refer the reader to  \cite[Proof of Propositions~1.7 and 1.10]{Kan17}.
\begin{proposition}[{\cite[Propositions~1.7 and 1.10]{Kan17}}]\label{prop:RCn-2}
 Let $(X,\sE)$ be a Mukai pair of dimension $n \geq 4$ and rank $n-2$ with $\rho_X = 1$.
 Then the following hold: 
\begin{enumerate}
 \item \label{prop:RCn-21} There exists an unsplit covering family $M$ of rational curves with anticanonical degree $l_X$.
 \item \label{prop:RCn-22} $X$ is chain connected by rational curves in the family $M$, that is, for each pair of points $x,y \in X$, there is a connected chain of rational curves in $M$ which contains both $x$ and $y$.
 \item \label{prop:RCn-23} If $l_X = n-2$, then $\sE$ is a direct sum of line bundles.
\end{enumerate}
\end{proposition}

Now we can prove $l_X =3$ for a pair as in Setting~\ref{set}.
\begin{proposition}\label{lem:lX}
Let $(X,\sE)$ be a Mukai pair as in Setting~\ref{set}.
Then $l_X=3$.
In particular, $r_X =1$ or $3$.
\end{proposition}

\begin{proof}
Note that $5 \geq l_X \geq 2$ by Lemma~\ref{lem:splitting}~\ref{lem:splitting2} and that $l_X\neq2$ by Proposition~\ref{prop:RCn-2}~\ref{prop:RCn-23}.

Assume to the contrary that $l_X  \geq 4 $.
Then $X \simeq \bP^4$ or $\bQ^4$ by Theorem~\ref{thm:PQ}.
It follows from the classification of Fano bundles of rank two on $\bP^n$ or $\bQ^n$ \cite{APW94} that the bundle $\sE$ is decomposable.
This contradicts our assumption that $\sE$ is indecomposable.
Thus $l_X = 3$.
Since the index $r_X$ divides the length $l_X$, we have $r_X =1$ or $3$.
\end{proof}

In what follows, we will fix a family $M$ of rational curves on $X$ as in Proposition~\ref{prop:RCn-2}~\ref{prop:RCn-21} for our Mukai pair $(X,\sE)$ as in Setting~\ref{set}.
A rational curve in this family is called a \emph{minimal rational curve}. 
By the above proposition, we have $l_X=3$.
Thus, if $f \colon \bP^1 \to C \subset X$ is a minimal rational curve on $X$, then $f^*\sE \simeq \cO(2) \oplus \cO (1)$ by Lemma~\ref{lem:splitting}~\ref{lem:splitting2}.
Corresponding to the surjection $\cO(2) \oplus \cO (1) \to \cO(1)$, a unique rational curve $\tf \colon \bP^1 \to \tC \subset \bP(\sE)$ exists such that $\xi_{\sE}\cdot\tC =1$ and the following diagram is commutative:
 \[
 \xymatrix{
 & \bP(\sE) \ar[d]^-\pi\\
 \bP^1 \ar[ru]^-\tf \ar[r]^-f & X.
 }
 \]
\begin{definition}[Minimal lifts]\label{def:lift}
 The above constructed rational curve $\tf \colon \bP^1 \to \tC \subset \bP(\sE)$ is called the \emph{minimal lift} of the minimal rational curve $f \colon \bP^1 \to C \subset X$.
\end{definition}

This construction can be globalized as follows:
let $(X,\sE)$ be a Mukai pair as in Setting~\ref{set} and $M$ a family of minimal rational curves on $X$.
Then we have the following diagram as in \eqref{dia:UF}:
\[
\xymatrix{
 U \ar[r]^-e \ar[d]^-p & X \\
 M.
}
\]

\begin{lemma}[Lifting map]\label{lem:lift}
 There exists a unique lift $\te \colon U \to \bP(\sE)$ of the evaluation map $e$ that restricts to the minimal lift on each minimal rational curve.
\end{lemma}
\begin{proof}
Let $\omega_p$ denote the relative canonical bundle of the universal family $p \colon U \to M$.
 Then the bundle $e^*\sE \otimes\omega_p$ is isomorphic to $\cO \oplus \cO (-1)$ on each $p$-fiber $\bP^1$.
 Thus we have the following exact sequence of \emph{vector bundles} on $U$
 \[
 0 \to p^*p_*e^*\sE \otimes \omega_p \to e^*\sE \otimes \omega_p \to Q \to 0,
 \]
 which restricts on each $p$-fiber to the sequence
  \[
 0 \to \cO \to \cO \oplus \cO(-1) \to \cO(-1) \to 0.
 \]
 This exact sequence gives the map $\te$ as desired.
\end{proof}

Finally, we include here a useful lemma.
\begin{lemma}\label{lem:BB}
Let $(X,\sE)$ be a pair as in Setting~\ref{set}.
Then there is no curve $D$ on $X$ such that $\dim\varphi(\pi^{-1}(D))=1$.
\end{lemma}
\begin{proof}
 Assume to the contrary that there is a curve $D$ on $X$ such that 
 \[\dim\varphi(\pi^{-1}(D))=1.\]
Since $\rho_X=1$, each minimal rational curve $f \colon \bP^1 \to C \subset X$ on $X$ is numerically proportional to $D$, and hence $\dim \varphi(\pi^{-1}(C)) =1$.
This contradicts to the fact that $f^*\sE \simeq \cO(2) \oplus \cO(1)$.
\end{proof}

\section{Stability and Bogomolov's inequality}\label{sect:stable}
In this section, we will check that the bundle $\sE$ is stable.
\begin{proposition}[Stability of $\sE$]
 Let $(X,\sE)$ be a Mukai pair as in Setting~\ref{set}.
 Then the bundle $\sE$ is stable.
\end{proposition}
As a consequence, we have the Bogomolov inequality for $\sE$, which is a crucial ingredient in our proof of Theorem~\ref{thm:main}.
Roughly speaking the inequality enables us to bound some invariants for the pair $(X,\sE)$.
Accordingly, the boundedness of invariants implies some finiteness results on possibilities of the structure of the pair.

First, we will recall the definition of stability of vector bundles, or torsion free sheaves.
Let $X$ be a smooth projective variety of dimension $n$ with an ample divisor $H_X$.
Given a torsion free sheaf $\sF$ on $X$, the \emph{slope} $\mu (\sF)$ with respect to $H_X$ is defined as 
\[
\mu (\sF) \coloneqq \frac{c_1(\sF)\cdot H_X^{n-1}}{\rank \sF}.
\]

A subsheaf $\sF_1 \subset \sF$ is said to be \emph{saturated} if the quotient $\sF/\sF_1$ is torsion free.

\begin{definition}[Stability of vector bundles]
A torsion free sheaf $\sF$ is called \emph{$H_X$-stable} if 
\begin{equation}\label{eq:slope}
 \mu(\sF_1)<\mu(\sF)
\end{equation}

holds for every nonzero saturated subsheaf $\sF_1 \subsetneq\sF$.
\end{definition}

In the following, we will restrict our attention to the case of a locally free sheave $\sF$ of rank two.
In this case, the stability of $\sF$ is easy to prove:
it suffices to check \eqref{eq:slope} for any saturated subsheaf $\sF_1$ of rank one.
Since $\sF$ is locally free, such a subsheaf $\sF_1$ is reflexive, and hence it is a line bundle \cite[Propositions~1.1 and 1.9]{Har80}.
Thus the stability of $\sF$ is equivalent to the vanishing
\[
H^0(\sF \otimes \sL^{-1}) =0
\]
for any line bundle $\sL$ with
\[
c_1(\sL)\cdot H_X^{n-1} \geq \frac{c_1(\sF)\cdot H_X^{n-1}}{2}.
\]

Moreover, if $\Pic(X)$ is generated by $\cO(H_X)$, then the above condition is equivalent to 
\[
H^0(\sF(-m))=0
\]
for any $m \geq \frac{c_1(\sF)}{2}$.

Let $(X,\sE)$ be a Mukai pair as in Setting~\ref{set}.
 Then, keeping with the above observation, the stability of $\sE$ follows from the next lemma. 
\begin{lemma}
  Let $(X,\sE)$ be a Mukai pair as in Setting~\ref{set}.
Then
\[
H^0(\sE(-m))=0
\]
for any $m \geq \frac{r_X}{2}$.

\end{lemma}

\begin{proof}
 Take a section $s \in H^0(\sE(-m))$ for $m \geq r_X /2$ and assume $s \neq 0$ to the contrary.
 Then by restricting to a general minimal rational curve $f \colon \bP^1 \to X$, we also have a nonzero section $f^*s$ of the bundle $f^*(\sE(-m))$, which is isomorphic to 
\[
\cO\left(2-\frac{3m}{r_X}\right)\oplus \cO\left(1-\frac{3m}{r_X}\right).
\]

Since the section $f^*s$ is nonzero, we have $2-\frac{3m}{r_X} \geq 0$, or equivalently, $\frac{2r_X}{3} \geq m$.
Note that $r_X=1$ or $3$ (Lemma~\ref{lem:lX}), and also that $m \geq \frac{r_X}{2}$ by our assumption.
Thus this is possible only if $r_X=3$ and $m=2$.
In this case, the bundle $f^*(\sE(-m))$ is isomorphic to $\cO\oplus \cO(-1)$.
Therefore, the section $f^*s$ is nowhere vanishing.
As $X$ is chain connected by  minimal rational curves (Proposition~\ref{prop:RCn-2}~\ref{prop:RCn-22}), the original section $s$ is also nowhere vanishing.
Hence we have an exact sequence 
\[
0 \to \cO \to \sE(-2) \to \sE(-2)/\cO \to 0
\]
of \emph{vector bundles} on $X$,
and the quotient $\sE(-2)/\cO$ is isomorphic to $\cO(-1)$.
This exact sequence, however, splits, since $\Ext ^1 (\cO(-1),\cO) =0$ by the Kodaira vanishing theorem.
This contradicts our unsplit assumption on $\sE$.
\end{proof}

As a consequence, the bundle $\sE$ is stable.
Set $\varDelta \coloneqq c_1(\sE)^2-4c_2(\sE)$.
Then the Bogomolov inequality yields the following.
\begin{proposition}[Bogomolov's inequality]\label{prop:Bogomolov}
 Let $(X,\sE)$ be a Mukai pair as in Setting~\ref{set}.
 Then
 \[
 \varDelta\cdot H_X^2 \leq0.
 \]
\end{proposition}
Note that, by the Grothendieck relation on the projective bundle $\bP(\sE)$,
we have
\begin{equation}\label{eq:Grel}
 \pi^*\varDelta = (2\xi_\sE-r_X\pi^*H_X)^2.
\end{equation}

\section{Rough description of the second contraction}\label{sect:rough}

Let $(X,\sE)$ be a pair as in Setting~\ref{set}, $M$ a family of minimal rational curves and $f \colon \bP^1 \to C \subset X$ a minimal rational curve.
Then the numerical equivalence class of the minimal lift $\tC$ does not depend on the choice of $C$;
the class $[\tC]$ is characterized by the numerical conditions $\xi_\sE \cdot [\tC] =1$ and $\pi^*(-K_X)\cdot [\tC] =3$.
\begin{definition}[Half line $\tR$]
 We will denote by $\tR$ the half line $\bR_{\geq0}[\tC]$ spanned by the class $[\tC]$.
\end{definition}

Since the Picard number of $\bP(\sE)$ is two, the nef cone $\NE(\bP(\sE))$ is spanned by two extremal rays; the ray $R_{\pi}$  corresponding to the projection $\pi$, and  the other ray $R_{\varphi}$ that defines the other extremal contraction $\varphi \colon \bP(\sE) \to Y$.
Note that $\tR \subset \NE(\bP(\sE))$:
\begin{figure}[h]
 \begin{picture}(100,100)(-20,0)
  \put(0,0){}
  \put(0,0){\line(1,0){70}}
  \put(0,0){\line(0,1){70}}
  \put(0,0){\line(1,1){60}}
  \put(70,-10){$R_\varphi$} 
  \put(-15,70){$R_{\pi}$}
  \put(60,70){$\tR \coloneqq \bR_{\geq0}[\tC]$}
\end{picture}
\caption{$\NE (\bP(\sE))$}
\end{figure}

\begin{remark}[$\tR$ and $R_\varphi$]\label{rem:comp}
\hfill 
\begin{enumerate}
\item A priori, it is not clear that the half line $\tR$ actually coincides with the extremal ray $R_\varphi$, or equivalently, that the minimal lifts $\tC$ are contracted by the second extremal contraction $\varphi$, cf. \cite[Lemma~3.1]{PSW92b}, \cite[Theorem~3.2]{Kan17}.
In our proof of Theorem~\ref{thm:main}, an important step consists of the proof of this assertion $\tR = R_{\varphi}$ by using the Bogomolov inequality (see Propositions~\ref{prop:Ineq} and \ref{prop:comp}).

\item In Section~8 of their paper \cite{NO07}, Novelli and Occhetta discussed the classification problem of Mukai pairs $(X,\sE)$ as in Setting~\ref{set} with an extra assumption $r_X \geq 2$.
Their argument proceeds as follows:
take a general element $Y$ of its fundamental linear system $|H_X|$.
By virtue of their assumption $r_X \geq 2$, $Y$ is a smooth Fano variety.
Then, they claimed that the restricted projective bundle $\bP(\sE|_Y)$ is also a Fano variety.
And, by using the classification results on $(Y,\sE|_Y)$, they recovered the structure of original $(X,\sE)$.
However, at least for the author, it is not clear that $\bP(\sE|_Y)$ is a Fano variety, as they claimed;
in fact, it is equivalent to the non-trivial assertion $\tR=R_\varphi$ (this equivalence follows from Proposition~\ref{prop:jumpingfib} below).

\end{enumerate}
\end{remark}

The dual cone of $\NE(\bP(\sE))$ is the nef cone $\Nef(\bP(\sE))$, which is spanned by two rays $\bR_{\geq0}[\pi^*H_X]$ and $\bR_{\geq0}[\varphi^* H_Y]$, where $H_Y$ is a divisor corresponding to the ample  generator of $\Pic(Y)$.
Since $\xi_\sE$ is ample on $\bP(\sE)$, we can find a positive rational number $a \in \bQ_{>0}$ such that $[a\xi_\sE - \pi^*H_X] \in \bR_{\geq0}[\varphi^* H_Y]$.
\begin{figure}[h]
 \begin{picture}(100,70)(-50,0)
  \put(0,0){}
  \put(0,0){\line(1,0){70}}
  \put(0,0){\line(-1,1){55}}
  \put(70,-10){$[\pi^*H_X]$} 
  \put(5,50){$\xi_\sE$}
  \put(-100,60){$[a\xi_\sE - \pi^*H_X]$}
  \multiput(0,0)(0,2){30}{\line(0,1){1}}
\end{picture}
\caption{$\Nef (\bP(\sE))$}
\end{figure}

\begin{lemma}[$ar_X \geq3$]\label{lem:arX}
Let the notation be as above.
Then $ar_X \geq3$ with equality holding if and only if $\tR = R_\varphi$
\end{lemma}
\begin{proof}
 Since $a\xi_\sE - \pi^*H_X$ is nef, we have
 \[
( a\xi_\sE - \pi^*H_X) \cdot \tC \geq 0,
 \]
 and the equality holds if and only if the curve $\tC$ is contracted by $\varphi$.
 The assertion is simply a rephrasing of this fact.
\end{proof}

First, by using Theorem~\ref{thm:IWineq} and \cite{Wat14b}, we show that there exists at least one $\varphi$-fiber whose dimension is not small.

\begin{proposition}\label{prop:jumpingfib}
 Let $(X,\sE)$ be a pair as in Setting~\ref{set}.
 Then there exists a $\varphi$-fiber $F$ with $\dim F \geq 2$.
\end{proposition}
\begin{proof}
By \cite[Theorem~1.1]{Wat14b}, the other contraction $\varphi$ is not a smooth $\bP^1$-bundle.

 Assume to the contrary that every $\varphi$-fiber has dimension one or zero.
Then, by Theorem~\ref{thm:IWineq}, the morphism $\varphi$ is of fiber type and every $\varphi$-fiber has dimension one.
Moreover, by adjunction, each general $\varphi$-fiber $F$ is isomorphic to $\bP^1$ and $\cO(\xi_\sE)|_F  \simeq \cO_{\bP^1}(1)$.
This implies that $\varphi$ is a smooth $\bP^1$-bundle \cite[Lemma~2.12]{Fuj87}, and we get a contradiction.
\end{proof}

On the other hand, as a first application of the Bogomolov inequality (Proposition~\ref{prop:Bogomolov}), we can now prove the following.

\begin{proposition}[Bounding the dimension of $\varphi$-fibers]\label{prop:bound}
Let $(X,\sE)$ be a pair as in Setting~\ref{set} and the notation as above.
Then $\dim F \leq 2$ for every $\varphi$-fiber $F$.
\end{proposition}

\begin{proof}
 Assume to the contrary that there is a $\varphi$-fiber $F$ with $\dim F \geq 3$.
 Take $F'$ a closed subvariety of $F$ with dimension three.
 Then $\pi (F')$ is a divisor in $X$ and, since $\rho_X =1$, the divisor $\pi(F')$ is numerically proportional to $H_X$.
 Thus the Bogolmolov inequality and the projection formula yields
 \[
 \pi^*\varDelta \cdot \pi^*H_X \cdot F' \leq 0.
 \]
 Then, by the Grothendieck relation \eqref{eq:Grel} for $\sE$, the above inequality is equivalent to
 \[
 \pi^*(2\xi_\sE -r_X\pi^*H_X)^2 \cdot \pi^*H_X \cdot F' \leq 0.
 \]
 Since $F'$ is contracted to a point by the second contraction $\varphi$, we have $a\xi_\sE \nequiv \pi^*H_X$ on $F'$.
 Thus the inequality gives
 \[
 a(ar_X-2)^2 \xi_\sE ^3 \cdot F' \leq 0.
 \]
 This contradicts the facts that $\xi_\sE$ is ample and $ar_X \geq 3$.
\end{proof}

As a corollary, we have the following.
\begin{corollary}[Length of extremal rays]\label{cor:length}
 Let $(X,\sE)$ be a pair as in Setting~\ref{set}.
 Then the following hold:
 
\begin{enumerate}
 \item\label{cor:length1} $l(R_\varphi) =2$.
 \item\label{cor:length2} There is a rational curve $C_\varphi$ on $\bP(\sE)$ such that $C_\varphi$ is contracted by $\varphi$ and $\xi_\sE \cdot C_\varphi =1$.
 \item\label{cor:length3} $a$ is an \emph{integer}.
 \item\label{cor:length4} There is an ample divisor $L$ on $\bP(\sE)$ such that $K_{\bP(\sE)} + 2 L$ is a supporting divisor for $\varphi$.
\end{enumerate}
\end{corollary}

\begin{proof}
Note that $l(R_\varphi) \in 2 \bZ$ since $-K_{\bP(\sE)}=2\xi_\sE$.
The first assertion follows from the Ionescu-Wi\'sniewski inequality (Theorem~\ref{thm:IWineq}) and the inequality $\dim F \leq 2$ (Proposition~\ref{prop:bound}).
The second one follows from the definition of $l(R_\varphi)$.
Let $C_\varphi$ be a curve as in \ref{cor:length2}.
Then $(a\xi_\sE - \pi^*H_X)\cdot C_\varphi =0$.
Therefore, $a =\pi^*H_X\cdot C_\varphi \in \bZ $.
This proves the third assertion.
By setting $L \coloneqq (\pi^*H_X\cdot C_\varphi +1)\xi_\sE - \pi^*H_X$, we have the last assertion.
\end{proof}

\begin{theorem}[Rough description of the other contraction]\label{thm:phi}
Let $(X,\sE)$ be a Mukai pair as in Setting~\ref{set}.
Then one of the following holds:
\begin{enumerate}
 \item \label{thm:phi2}
 The morphism $\varphi$ is a quadric bundle over a smooth projective $3$-fold $Y$.
 \item \label{thm:phi3}
 The morphism $\varphi$ is a special \Banica scroll over a smooth projective $4$-fold $Y$.
\end{enumerate} 
\end{theorem}
\begin{proof}
Let $E$ be an irreducible component of $\Exc(\varphi)$ and $F$ a fiber contained in $E$.
Then, by the Ionescu-Wi\'sniewski inequality (Theorem~\ref{thm:IWineq}), we have
\[
\dim E + \dim F \geq \dim \bP (\sE) + l(R_\varphi) -1.
\]
Since $\dim F \leq 2$ (Proposition~\ref{prop:bound}) and $l(R_\varphi ) = 2$, the above inequality yields
\[
\dim E + 2 \geq \dim E + \dim F \geq 6.
\]
In particular, the morphism $\varphi$ is either divisorial or of fiber type.
Hence $E = \Exc (\varphi)$.

First assume that $\varphi$ is of fiber type.
Then, since $\dim F \leq 2$, we have $\dim Y = 3$ or $4$.
If $\dim Y = 3$, then $\dim F = 2$ for any $\varphi$-fiber $F$ and hence \ref{thm:phi2} holds by Theorem~\ref{thm:Qbundle}.
On the other hand, if $\dim Y =4$, then $\varphi$ is a scroll and \ref{thm:phi3} holds by Theorem~\ref{thm:Bscroll}.

Next assume that $\varphi$ is divisorial.
Then $\dim F =2$ for any non-trivial fiber $F$.
Then, by Theorem~\ref{thm:Bl}, the morphism $\varphi$ is obtained by blowing up a smooth projective $5$-fold $Y$ along a smooth surface $S \subset Y$.\[
\xymatrix{
&W\ar[ld]_-\pi \ar[rd]^-\varphi &E \coloneqq \Exc(\varphi) \ar@{_{(}-_>}[l] \ar[rd]&\\
X&&Y&S.\ar@{_{(}-_>}[l]
}
\]
Thus $W$ is a smooth Fano $5$-fold with index two that has a birational contraction $\varphi$ as above and admits a contraction $\pi$ of fiber type.

Thus this Fano $5$-fold $W$ fits Case (b) of \cite{CO08}.
One can deduce from their classification that there is no Fano $5$-fold $W$ as above.
Thus $\varphi$ is not birational.
\end{proof}

\begin{remark}
In the above proof, we used the classification result of Case (b) in \cite{CO08}, which is carried out in Section~4 of their paper.
Note that their argument in \cite[Section~4]{CO08}, more precisely, the proof of \cite[Lemma~4.4]{CO08} relies on \cite[Theorem~1.3]{NO07},
which, as we mentioned, seems to contain a gap in the proof.
Nevertheless, their usage of \cite[Theorem~1.3]{NO07} is mild; they only use the fact that the minimum anticanonical degree of rational curves on $X$ is two or three, which we already have proved.
Thus we can avoid a circular argument.
\end{remark}

\section{Case: $\varphi$ is a quadric bundle}\label{sect:quarticDP}
As before, $(X,\sE)$ is a pair as in Setting~\ref{set}.
Then, by Theorem~\ref{thm:phi}, $\varphi$ is either a quadric bundle or a special \Banica scroll.
In this section, we will deal with the case of quadric bundle, and prove the following theorem.
\begin{theorem}\label{thm:quarticDP}
 Let $(X,\sE)$ be a pair as in Setting~\ref{set}.
 Assume that the second contraction $\varphi$ is a quadric bundle.
 Then $(X ,\sE)$ is isomorphic to 
 \[
 (V_4, p^*\sS^* _{\bQ^4} (1)),
 \]
 where $V_4$ is a quartic del Pezzo $4$-fold obtained as a double cover $p \colon V_4 \to \bQ ^4$ branched along a smooth divisor $B \in \left| \cO_{\bQ^4}(2) \right|$, and $\sS_{\bQ^4}$ is a spinor bundle on $\bQ^4$.
\end{theorem}

The rest of this section is occupied with the proof of Theorem~\ref{thm:quarticDP}.
First, we will determine the isomorphic class of the image $Y$ of the second contraction $\varphi$.
\begin{lemma}
Let $(X,\sE)$ be a pair as in Theorem~\ref{thm:quarticDP}.
Then $Y$ is isomorphic to $\bP^3$.
\end{lemma}
\begin{proof}
Let $F$ be a general $\varphi$-fiber.
Then by taking a base change of the projection $\pi$ by the morphism $\pi|_F$, we have the following diagram: 
\begin{equation}\label{diagram:fiber}
\vcenter{
\xymatrix{
  \bP(\sE |_{F}) \ar[d]_-{\pi_{F}}  \ar[r] \ar@(ur,ul)[rr]^-{\theta _{F}}  & \bP(\sE) \ar[d]_-{\pi} \ar[r]_-{{\varphi}} & Y         \\
F \ar[r]^-{\pi|_{F}} & X.                                        & \\
}
}
\end{equation}
The morphism $\theta _F$ is surjective.
Otherwise $\dim  \theta_F (\bP(\sE|_F)) < \bP(\sE|_F)$.
Then, by the bend and break lemma \cite[Chapter II, Theorem~5.4]{Kol96}, there is a curve $D \subset F$ such that $\dim \theta_F (\pi_F^{-1}(D)) =1$.
This contradicts Lemma~\ref{lem:BB}.

By adjunction, $F$ is isomorphic to $\bP^1 \times \bP^1$ and
\[
\cO(\xi_\sE)|_F \simeq \pr_1^*\cO(1) \otimes \pr_2^*\cO(1),
\]
where $\pr_i \colon \bP^1 \times \bP^1 \to \bP^1$ is the projection to the $i$-th factor.

By restricting the relative Euler sequence of the projectivization $\bP(\sE) \to X$, we have the following exact sequence:
\[
0 \to \cO(K_\pi +\xi _\sE)|_F  \to \sE|_F \to \pr_1^*\cO(1) \otimes \pr_2^*\cO(1) \to 0. \]

Since $K_\pi \simeq -2 \xi_\sE+ r_X \pi^*H_X$, the bundle $\cO(K_\pi +\xi _\sE)|_F$ is isomorphic to
\[
\cO(r_X \pi^*H_X)|_F \otimes  \pr_1^*\cO(-1) \otimes \pr_2^*\cO(-1).
\]
Thus the class of the above exact sequence belongs to
\[
\Ext^1(\pr_1^*\cO(1) \otimes \pr_2^*\cO(1),\cO(r_X \pi^*H_X)|_F \otimes  \pr_1^*\cO(-1) \otimes \pr_2^*\cO(-1)),
\]
which is isomorphic to
\[
H^1(K_F\otimes \cO(r_X \pi^*H_X)|_F).
\]
Since $F$ is a $\varphi$-fiber, the map $\pi|_F \colon F \to X$ is finite, and hence the divisor $\pi^*H_X|_F$ is ample.
Thus, by the Kodaira vanishing theorem, the cohomology group $
H^1(K_F\otimes \cO(r_X \pi^*H_X)|_F)$ is trivial, and the above exact sequence splits.
Therefore, $\bP(\sE|_F)$ is a toric variety.
Now the assertion follows from \cite{OW02}.
\end{proof}

Second, we will determine the values $a$, $r_X$ and $H_X^4$. 
Recall that $a$ is the integer for which the divisor $a\xi_\sE -\pi^*H_X$ is a supporting divisor of the contraction $\varphi$.
\begin{lemma}
Let $(X,\sE)$ be a pair as in Theorem~\ref{thm:quarticDP}.
Then $a=1$, $r_X =3$ and $H_X^4 =4$.
\end{lemma}
\begin{proof}
 We will denote by $H_Y$ the divisor class of a hyperplane in $Y \simeq \bP^3$.
 Since the divisor $a\xi_\sE -\pi^*H_X$ is not a multiple of another divisor, we have
 \[
 a\xi_\sE -\pi^*H_X = \varphi^*H_Y.
 \]
Thus
\[
 (a\xi_\sE -\pi^*H_X)^3\cdot\xi_\sE^2 
=2
\text{ and }
(a\xi_\sE -\pi^*H_X)^4=0.
\]
Set
\begin{align*}
 x&\coloneqq H_X^4,\\
 y&\coloneqq c_2(\sE)\cdot H_X^2,\\
 z&\coloneqq c_2(\sE)^2.
\end{align*}
By a direct calculation with the Grothendieck relation
\[
 \xi_\sE^2=\xi_\sE\cdot (r_X\pi^*H_X) - \pi^*c_2(\sE),
\]
 the above equations yield
\begin{align*}
(r_X^4a^3-3r_X^3a^2+3r_X^2a-r_X)x-(3r_X^2a^3-6r_Xa^2+3a)y+a^3z= 2,\\
(r_X^3a^4-4r_X^2a^3+6r_Xa^2-4a)x+(4a^3-2r_Xa^4)y=0,\\
x-(r_X^2a^4-4r_Xa^3+6a^2)y+a^4z= 0.
\end{align*}
Solving this system of linear equations for $(x,y,z)$, we have 
\begin{align}
x&=\frac{4a}{(r_Xa-2)^2}, \label{eq:Q1}\\
y &= \frac{2r_X^2a^2-4r_Xa +4}{a(r_Xa-2)^2},\label{eq:Q2}\\
z &=\frac{2r_X^4a^4-12r_X^3a^3+32r_X^2a^2-40r_Xa+20}{a^3(r_Xa-2)^2}.\label{eq:Q3}
\end{align}
Note that, by Lemmas~\ref{lem:lX} and \ref{lem:arX}, we have $r_X=1$ or $3$, and $r_Xa \geq 3$.
Assume $r_X =1$. Then, by \eqref{eq:Q1}, we have
\[
\frac{4a}{(a-2)^2} \in \bZ.
\]
Thus $a = 3$ or $4$.
In each case, however, we have $y = 5/4$ or $10/3$ respectively, which is impossible since $y= c_2(\sE)\cdot H_X^2 \in \bZ$.
Thus we have $r_X =3$.
Then, since $x=\frac{4a}{(3a-2)^2} \in \bZ$, we have $a=1$.
This completes the proof.
\end{proof}

\begin{proof}[Proof of Theorem~\ref{thm:quarticDP}]
Since $\xi_\sE-\pi^*H_X =\varphi^*H_Y$, each $\pi$-fiber maps to a line in $Y \simeq \bP^3$, and, by the universality of Grassmannian variety, we have the following commutative diagram:
\[
\xymatrix{
  \bP(\sE ) \ar[d]_-{\pi} \ar[r] \ar@(ur,ul)[rr]^-{\varphi}  & \bP(\sS ^* _\bQ) \ar[d]_-{\pi_\bQ^4} \ar[r] & \bP^3         \\
X \ar[r]^-{p} & \bQ^4 \simeq \Gr(2,4),& \\
}
\]
with the condition $p^*\sS_{\bQ^4}^* \simeq \sE(-1)$.
In particular $p^*\cO_{\bQ^4}(1) \simeq \cO_X(1)$.

Since $\rho_X=1$, the morphism $p$ is finite and surjective.
Moreover, the covering $p$ is of degree two since $H_X ^4 =4$.
Then, by adjunction, we see that the double covering $p$ is branched along a smooth divisor $B \in \left| \cO_{\bQ^4}(2)\right|$.
This completes the proof.
\end{proof}

\section{Case: $\varphi$ is a special B\u anic\u a scroll}\label{sect:quinticDP}
Let $(X,\sE)$ be a pair as in Setting~\ref{set}.
In this final section, we will deal with the case where the second contraction $\varphi$ is a special \Banica scroll, and thus complete the proof of Theorem~\ref{thm:main}.

\begin{theorem}[Special \Banica scroll]\label{thm:quinticDP}
Let $(X,\sE)$ be a pair as in Setting~\ref{set}.
Assume that $\varphi$ is a special \Banica scroll.
Then the pair $(X,\sE)$ is isomorphic to
\[
(V_5, \sS ^*_{V_5} (1)),
\]
where $V_5$ is a general linear section of the Grassmannian $\Gr (2,5)$ embedded into $\bP ^9$ via the Pl\"ucker embedding, and $\sS _{V_5}$ is the restriction of the universal subbundle on $\Gr(2,5)$ to $V_5$.
\end{theorem}

First, we will set up the notations. 
\begin{notation}[$\sF$, $E$, $B$, $d_i$, $\tW$, $\tY$, $\tE$, $\tB$]
 From now on, we will assume that $(X,\sE)$ is a pair as in Theorem~\ref{thm:quinticDP}.
 Recall that $W$ denotes the projectivization $\bP(\sE)$.

Set $\sF \coloneqq \varphi _* \cO(\xi_\sE)$.
Then, by Theorem~\ref{thm:Bscroll}, the sheaf $\sF$ is a \Banica sheaf and the morphism $\varphi$ is the projection of the projectivization $\bP(\sF) \to Y$.
Then general $\varphi$-fibers are isomorphic to $\bP^1$, and the other fibers $\bP^2$;
let 
\[
B \coloneqq \{\, y\in Y \mid \varphi^{-1}(y) \simeq \bP^2 \,\}
\]
 be the locus parametrizing jumping fibers and $E \coloneqq \varphi^{-1}(B)$ the preimage of $B$.
 Then, by Proposition~\ref{prop:Bscroll}, $Y$ is smooth, $B$ is a disjoint union of smooth irreducible curves $B_i$ for $i = 1, \dots, m$, and the morphism $E \to B$ is a $\bP^2$-bundle.
Set $E_i \coloneqq \varphi^{-1}(B_i)$.

Let $\tW \to W$ (resp.\ $\tY \to Y$) be the blowing-up of $W$ (resp.\ $Y$) along $E$ (resp.\ $B$), and $\tE_i$ (resp.\ $\tB_i$) the exceptional divisor over $E_i$ (resp.\ $B_i$).
Then, by Proposition~\ref{prop:Bscroll} again, we have the following commutative diagram:
 \[
 \xymatrix{
 &\tE = \coprod \tE_i \ar@{_{(}-_>}[ld] \ar[dd] \ar[rr]^-{\tvarphi_{\tE}=\coprod\tvarphi_{\tE_i}}&&\tB= \coprod \tB_i\ar@{_{(}-_>}[ld] \ar[dd]\\
 \tW \ar[rr]^(0.6){\tvarphi} \ar[dd] && \tY \ar[dd]&\\
 &E = \coprod E_i \ar@{_{(}-_>}[ld] \ar[rr]&  & B= \coprod B_i\ar@{_{(}-_>}[ld]\\
 W \ar[d]^-\pi \ar[rr]^-\varphi && Y &\\
 X&&&
 }
 \]
with the following conditions:
\begin{enumerate}
 \item $\tvarphi \colon \tW \to \tY$ is a $\bP^1$-bundle,
 \item $\tE = \tvarphi^{-1}(\tB)$.
\end{enumerate}

As $Y$ is smooth uniruled projective manifold with Picard number one, it is a Fano manifold with Picard number one.
As usual, $H_Y$ is the ample generator of $\Pic (Y)$, $r_Y$ the index of $Y$ and $l_Y$ the length of $Y$.
Let $d_i \coloneqq H_Y \cdot B_i$ be the degree of $B_i$ with respect to $H_Y$.
Note that $\xi_\sE = \xi_\sF$ on $W$.

\end{notation}

In the following lemma, we will see that the pair $(Y,\sF)$ is nearly a Mukai pair; the sheaf $\sF$ is not locally free, but the pair satisfies conditions in the definition of Mukai pairs.
\begin{lemma}
 Let $(X,\sE)$ be a pair as in Theorem~\ref{thm:quinticDP}.
Then
\begin{enumerate}
 \item \label{lem:Mukai1}$\sF$ is ample and $c_1(Y) = c_1(\sF)$.
 \item \label{lem:Mukai2}$l_Y \geq 3$.
\end{enumerate}
\end{lemma}
\begin{proof}
 \ref{lem:Mukai1} Recall that $\sF$ is said to be ample if $\xi_{\sF}$ is ample.
 Thus the ampleness of $\sF$ follows from the definition.
 Set $W^0 \coloneqq W \setminus E$ and $Y^0 \coloneqq Y \setminus B$, and let $\varphi^0 \colon W^0 \to Y^0$ be the restriction of the projection.
 Then the canonical bundle formula for projective bundles yields 
 \[
 c_1(W^0 ) = 2 \xi_{\sF|_{Y^0}}+(\varphi^0)^* (c_1(Y^0)-c_1(\sF|_{Y^0})).
 \]
 Since
 \[
 c_1(W^0)=2\xi_{\sE}|_{W^0} = 2 \xi_{\sF}|_{W^0}=2 \xi_{\sF|_{Y^0}},
 \]
 we have $c_1 (\sF|_{Y^0}) =c_1(Y^0)$.
 Since $\codim _Y B \geq 2$, we have $c_1(\sF) = c_1 (Y)$.
 
 \ref{lem:Mukai2} Assume that $l_Y =2$.
Then $Y$ admits a dominating family $M_Y$ of rational curves of anticanonical degree $2$.
Let $g \colon \bP^1 \to C_Y \subset Y$ be a general rational curve in this family $M_Y$.
Then, by \cite[Chapter~II, Proposition~3.7]{Kol96}, the image $C_Y$ does not intersect with $B$.
Thus $g^*\sF$ is a locally free sheaf of rank two on $\bP^1$.
Moreover, by \ref{lem:Mukai1}, the bundle $g^*\sF$ is isomorphic to $\cO(1)\oplus\cO(1)$.
Therefore, by considering the minimal sections of $\bP(g^*\sF) \to \bP^1$, we have an unsplit covering family $\tM_Y$ of rational curves $\tC_Y$ on $W$ such that $\xi_\sF \cdot \tC_Y =1$ (this family is unsplit by the numerical condition $\xi_\sF \cdot \tC_Y =1$).
By \cite[Lemma~2.4]{CO06}, the numerical equivalence class $[\tC_Y]$ spans the extremal ray $R_\pi$.
Thus we have $\dim \pi (\varphi^{-1}(C_Y)) =1$.
This implies $\pi^{-1}(\pi (\varphi^{-1}(C_Y))) = \varphi^{-1}(C_Y)$, and hence $\varphi(\pi^{-1}(\pi (\varphi^{-1}(C_Y)))) = C_Y$.
This contradicts Lemma~\ref{lem:BB}.
\end{proof}

Now we can determine the curves $B_i$.
\begin{lemma}
 Let $(X,\sE)$ be a pair as in Theorem~\ref{thm:quinticDP}.
Then each curve $B_i$ is isomorphic to $\bP^1$.
\end{lemma}
\begin{proof}
It is enough to see that the first Betti number $b_1(B_i)$ is zero for each $i$.

 By the blow-up formula and the projective bundle formula for Betti numbers (see for example \cite[EXPOSE XVIII]{SGA72}), we have
\begin{align}
 b_k(\tY) &= b_k(Y)+\sum_{i=1}^m b_{k-2} (B_i)+\sum_{i=0}^m b_{k-4} (B_i), \label{eq:betti1}\\
  b_k(\tW) &= b_k(W)+\sum_{i=1}^m b_{k-2} (E_i).\label{eq:betti2}
\end{align}
and
\begin{align}
 b_k(\tW) &= b_k(\tY)+ b_{k-2} (\tY),\label{eq:betti3}\\
  b_k(W) &= b_k(X)+b_{k-2} (X),\label{eq:betti4}\\
  b_k(E_i) &= b_k(B_i)+ b_{k-2} (B_i)+b_{k-4} (B_i)\label{eq:betti5}.
\end{align}
By equations~\eqref{eq:betti2}, \eqref{eq:betti4} and \eqref{eq:betti5}, we have
\[
b_k(\tW) = b_k(X)+b_{k-2} (X) + \sum_{i=1}^m (b_{k-2}(B_i)+ b_{k-4} (B_i)+b_{k-6} (B_i)).
\]
Also, by equations~\eqref{eq:betti1} and \eqref{eq:betti3}, we have
\[
b_k(\tW) = b_k(Y)+b_{k-2} (Y) + \sum_{i=1}^m (b_{k-2}(B_i)+ 2 b_{k-4} (B_i)+b_{k-6} (B_i)).
\]
Hence,
\[
b_k(X)+b_{k-2} (X) = b_k(Y)+b_{k-2} (Y) + \sum_{i=1}^m  b_{k-4}(B_i).
\]
By letting $k=3$, $4$ or $5$, we obtain:
\begin{align*}
 b_3(X) &=b_3(Y),\\
  b_4(X)+b_{2} (X) &= b_4(Y)+b_{2} (Y) + \sum_{i=1}^m  b_{0}(B_i),\\
 b_5(X)+b_{3} (X) &= b_5(Y)+b_{3} (Y) + \sum_{i=1}^m  b_{1}(B_i).
\end{align*}
Note that $b_k(X) = b_{8-k}(X)$ and $b_k(Y) = b_{8-k} (Y)$ by the Poincar\'e duality, and that $b_2(X) = b_2 (Y) =1$.
Thus, by solving the above three equations, we have $b_1(B_i)=0$, $b_3(X) =b_3(Y)$ and $b_4(Y) +m =b_4(X)$, and the assertion follows
\end{proof}

In the following, $\sC_{{\blank/\blank}}$ (resp.\ $\sN_{\blank/\blank}$) denotes the conormal (resp.\ normal) bundle.
\begin{lemma}
 Let $(X,\sE)$ be a pair as in Theorem~\ref{thm:quinticDP}.
Then $E_i$ is isomorphic to $\bP(\sN_{B_i/Y}(1))$ over $B_i \simeq \bP^1$, and, via this identification, we have
 \[
 \xi_{\sE}|_{E_i} \simeq \xi _{\sN_{B_i/Y}(1)}.
 \]
\end{lemma}
\begin{proof}
In the following square
 \[
 \xymatrix{
  \tE_i \ar[d]_-f \ar[r]^-{\tvarphi_{\tE_i}}&\tB_i \ar[d]^-g\\
 E_i \ar[r]  & B_i,\\
 }
 \]
each $f$-fiber gives a line in a $g$-fiber, and $E_i$ parametrizes all the lines in $g$-fibers.
Since $\tB_{i}\simeq \bP(\sC_{B_i/Y})$, the parameter space of lines in $g$-fibers is isomorphic to the dual projectivization $\bP(\sN_{B_i/Y})$.
Thus we have the first assertion.

Note that $\tvarphi^*\tB_{i} = \tE_i$.
 Thus
\begin{align*}
(\tvarphi_{\tE_i})^*\xi_{\sC_{B_i/Y}} &\simeq  (\tvarphi_{\tE_i})^*(-\tB_i|_{\tB_i}) \\
  &\simeq -\tE_i|_{\tE_i}\\
  &\simeq \xi_{\sC_{E_i/W}}.
\end{align*}

Since $-K_W = 2 \xi_{\sE}$, we have $-K_{\tW} = 2\xi_{\sE}|_{\tW}-\tE$.
Thus, by adjunction,
\[
-K_{\tE_i} = 2\xi_\sE|_{\tE_i} + 2 \xi_{\sC_{E_i/W}}.
\]
Then, we have 
\[
-K_{\tE_i} = 2\xi_\sE|_{\tE_i} + 2 \xi_{\sC_{B_i/Y}}|_{\tE_i}.
\]

On the other hand, we have
\[
-K_{\tE_i} = (2\xi_{\sN_{B_i/Y}} + 2 \xi_{\sC_{B_i/Y}} + 2H_{B_i})|_{\tE_i},
\]
where $H_{B_i}$ is the ample generator of $\Pic (B_i)$ (note that $B_i \simeq \bP^1$).

Thus
\[
\xi_\sE|_{\tE_i} = (\xi_{\sN_{B_i/Y}} +  H_{B_i})|_{\tE_i}.
\]
Hence
\[
\xi_\sE|_{E_i} =\xi_{\sN_{B_i/Y}} +  H_{B_i}|_{E_i} = \xi_{\sN_{B_i/Y}(1)}.
\]
\end{proof}

As a second application of Bogomolov's inequality (Proposition~\ref{prop:Bogomolov}), we now prove the following lemma, which enables us to control invariants $a$ and $r_X$.

\begin{proposition}\label{prop:Ineq}
 Let $(X,\sE)$ be a pair as in Theorem~\ref{thm:quinticDP}.
Then
\[
a < \frac{3r_X+2r_Y+\sqrt{9r_X^2+4r_Yr_X+4r_Y^2}}{2r_Yr_X}.
\]
In particular, one of the following holds:
\begin{enumerate}
 \item $ar_X =3$ and $r_Y\geq2$,
 \item $a \leq 3$, $r_X=3$ and $r_Y =1$,
 \item $a \leq 4$, $r_X=1$ and $r_Y=1$.
\end{enumerate}
\end{proposition}
\begin{proof}

 Since $\pi_{*}(E_i)$ is an effective divisor (possibly zero), we have
 \[
0 \geq E_i\cdot\pi^*H_X\cdot\pi^*\varDelta
 \]
 by the Bogomolov inequality (Proposition~\ref{prop:Bogomolov}) and the projection formula.
 This inequality yields
 
\begin{align*}
 0 &\geq E_i\cdot\pi^*H_X\cdot\pi^*\varDelta\\
 &= E_i\cdot\pi^*H_X\cdot(2\xi_\sE-r_X\pi^*H_X)^2\\
&= E_i\cdot(a\xi_\sE-\varphi^*H_Y)\cdot((ar_X-2)\xi_\sE-r_X\varphi^*H_Y)^2\\
&= (a\xi_\sE-\varphi^*H_Y)|_{E_i}\cdot((ar_X-2)\xi_\sE-r_X\varphi^*H_Y)^2|_{E_i}\\
&= (a\xi_{\sN_{B_i/Y}}+(a-d_i)H_{B_i}|_{E_i})\cdot((ar_X-2)\xi_{\sN_{B_i/Y}}+(ar_X-d_ir_X-2)H_{B_i}|_{E_i})^2.
\end{align*}

Note that $c_1(\sN_{B_i/Y}) = d_i r_Y -2$, and hence $\xi_{\sN_{B_i/Y}}^3 =d_i r_Y-2$.
Thus, by a straightforward calculation, we have
\begin{align*}
0 &\geq(r_Xa-2)\left((d_ir_Xr_Y+r_X)a^2-(3d_ir_X+2d_ir_Y+2)a+2d_i \right)\\
&=(r_Xa-2)\left(d_i\left(r_Xr_Ya^2-(3r_X+2r_Y)a+2\right) +a\left(r_Xa-2\right) \right).
\end{align*}
Since $ar_X \geq 3$, $a>0$ and $d_i>0$, the above inequality yields the following quadratic inequality for $a$:
\[
0 > r_Xr_Ya^2-(3r_X+2r_Y)a+2.
\]
This gives the inequality as claimed.

Note that we have only finite possibilities for $r_X$ and $r_Y$, and that $ar_X \geq3$.
The rest of the assertion follows from case-by-case analysis.
\end{proof}

\begin{proposition}[$ar_X=3$]\label{prop:comp}
 Let $(X,\sE)$ be a pair as in Theorem~\ref{thm:quinticDP}.
Then $ar_X =3$, or equivalently $\tR=R_\varphi$.
\end{proposition}
\begin{proof}
 Assume to the contrary $ar_X >3$, or equivalently, $\tR \neq R_\varphi$.
 Then, by Proposition~\ref{prop:Ineq}, one of the following holds:
\begin{enumerate}
 \item $a=3$, $r_X=3$ and $r_Y=1$,
 \item $a=2$, $r_X=3$ and $r_Y=1$,
 \item $a=4$, $r_X=1$ and $r_Y=1$.
\end{enumerate}
In any case, we have $\varphi^*(-K_Y)\cdot \tC = a -\frac{3}{r_X}\leq 2$.
Thus the $\varphi$-image of $\tC$ gives a rational curve on $Y$ whose anticanonical degree is at most two.
Since $l_Y =3$, the family of these rational curves is not dominant. 
Hence $\varphi(\te(U))\neq Y$.
This also implies $\te(U) \neq W$.

Since $\pi(\te(U)) =X$, we have $\dim \te(U) \geq 4$.
Therefore, $D \coloneqq \te(U)$ is a divisor on $W$.
Then $\varphi(D)$ is also a divisor on $Y$, and $\varphi^*(\varphi(D)) = D$.
In this case, since $\rho _Y=1$, the divisor $\varphi(D)$ is ample, and hence $\varphi(D)\cap B_i \neq \emptyset $.
Thus $D$ contains a two dimensional fiber $F$ of $\varphi$.
Then, by \cite[Lemma~5.4]{ACO04}, \cite[Lemma~3.2 and Remark~3.3]{Occ06} (see also \cite[Corollary~2.2 and Remark~2.4]{CO06}),
\[
\dim \te(\pi^{-1}(\pi(\te^{-1}(F)))) \geq 4
\]
and
\[
\NE(\te(\pi^{-1}(\pi(\te^{-1}(F)))),W) \subset \langle \tR,R_{\varphi}\rangle.
\]
This contradicts the following lemma \cite[Lemma~1.18]{Kan17}, which is a generalization of \cite[Claim 4.1.1]{PSW92b}.
Note that the original assumption in \cite[Lemma~1.18]{Kan17} is a little bit stronger than that in our case, while the same proof does work. 
\end{proof}

\begin{lemma}[{\cite[Lemma~1.18]{Kan17}}]\label{lem:section}
 Let $(X,\sE)$ be a Mukai pair as in \ref{set}.
 Then there is no subvariety $V$ in $W$ with $\dim V \geq 4$ such that
 $\NE (V,W) \subset \langle \tR,R_{\varphi} \rangle$.
\end{lemma}
\begin{proof}
Assume to the contrary that there is a closed subvariety $V$ in $W$ with $\dim V \geq 4$ and
 $\NE (V,W) \subset \langle \tR,R_{\varphi} \rangle$.
By the same proof of \cite[Lemma~1.18]{Kan17}, we can show that the subvariety $V$ gives a section of $\pi$ corresponding to the exact sequence
\[
0 \to \cO(2) \to \sE \to \cO(1) \to 0.
\]
This sequence, however, splits, and we get a contradiction.
\end{proof}

Therefore, by Proposition~\ref{prop:comp}, the minimal lifts $\tC$ are contracted by the second contraction $\varphi$.
Conversely, each line in a $\varphi$-fiber  gives a minimal rational curve on $X$.
Using this fact, we can now prove that the image $Y$ is $\bP^4$.
\begin{proposition}
 Let $(X,\sE)$ be a pair as in Theorem~\ref{thm:quinticDP}.
 Then $Y \simeq \bP^4$.
\end{proposition}
\begin{proof}
By taking the base change of $\pi$ by the morphism $\pi|_{E_i} \colon E_i \to X$, we have the following diagram:
\[
\xymatrix{
\bar {E_i} \ar@{^{(}-_>}[d]^-{\iota} \ar[rr]&&B_i \ar@{^{(}-_>}[d]\\
  \bP(\sE |_{E_i}) \ar[d]_-{\pi_{E_i}}  \ar[r] \ar@(ur,ul)[rr]^-{\theta _{E_i}}  & W \ar[d]_-{\pi} \ar[r]_-{{\varphi}} & Y         \\
E_i \ar[r]^-{\pi|_{E_i}} & X,                                        & \\
}
\]
where $\bar {E_i}$ is a section of $\pi_{E_i}$ corresponding to the original inclusion $E_i \subset W$.

First, we will prove that $\theta_{E_i}$ is surjective.
It is enough to see that $\pi|_{E_i}$ is generically finite.
Otherwise, $\dim \pi(E_i) =2$ and thus $\pi^{-1}(\pi(E_i))=E_i$.
Fix a line $\ell$ in a $\varphi$-fiber contained in $E_i$.
Then $\varphi(\pi^{-1}(\pi(\ell)))=B_i$ since $\pi^{-1}(\pi(E_i))=E_i$.
This is impossible since ${(\pi|_{\ell})}^*\sE \simeq \cO(2)\oplus\cO(1)$.
Thus $\theta_{E_i}$ is surjective.

Second, we will prove that, for each jumping fiber $F \simeq \bP^2$ of $\varphi$, we have $\sE|_F \simeq \cO(2)\oplus \cO(1)$.
Take a jumping fiber $F\simeq \bP^2$.
Then, by restricting the Euler sequence for $\pi$, we have the following exact sequence:
\[
0 \to \cO(2) \to \pi^*\sE|_F \to \cO(1) \to 0.
\]
The assertion is now clear.

Finally, we will check that $Y \simeq \bP^4$.
Set $q \coloneqq \pi_{E_i} \circ \varphi|_{E_i} \colon  \bP(\sE |_{E_i}) \to B_i$.
Now each $q$-fiber is isomorphic to $\bP_{\bP^2}(\cO(2)\oplus\cO(1))$, and the minimal section of this projectivization $\bP_{\bP^2}(\cO(2)\oplus\cO(1))$ is contracted to a point by $\theta_{E_i}$.
Thus the map $\theta_{E_i}$ factors through a $\bP^3$-bundle $\bP(\sG_i)$ over $B_i$:
\[
\xymatrix{
\bP(\sE|_{E_i}) \ar[d]^-q \ar[r] \ar@(ur,ul)[rr]^-{\theta _{E_i}} & \bP(\sG_i) \ar[dl] \ar[r] & Y \\
B_i.&&
}
\]

Since every vector bundle on $B_i$ splits into a direct sum of line bundles, every projective bundle over $B_i$ is a toric variety.
Thus, by \cite{OW02}, $Y$ is isomorphic to $\bP^4$.
\end{proof}

As $Y \simeq \bP^4$, we now can calculate several invariants of $(X,\sE)$ by relating to those of $Y$.
Set
\begin{align*}
 x&\coloneqq H_X^4,\\
 y&\coloneqq c_2(\sE)\cdot H_X^2,\\
 z&\coloneqq c_2(\sE)^2,\\
 u&\coloneqq c_2(X)\cdot H_X^2,\\
 v&\coloneqq c_2(X)\cdot c_2(\sE).
\end{align*}
\begin{proposition}\label{prop:RR}
 Let $(X,\sE)$ be a pair as in Theorem~\ref{thm:quinticDP}.
Then $a=1$ and $r_X = 3$.
Moreover, we have
\begin{align*}
 y&= \frac{5x-1}{2},\\
 z&= \frac{13x-7}{2},\\
 u&= 2x+12,\\
 v&= \frac{13x+41}{2}.
\end{align*}

\end{proposition}
\begin{proof}
By the the Hirzebruch-Riemann-Roch theorem,
\begin{equation}\label{eq:RR}
 \chi(m(a\xi_\sE-\pi^*H_X))  = [\ch(m(a\xi_\sE-\pi^*H_X))\cdot\td(W)]_{\dim W}.
\end{equation}

Since $\chi(\cO_X) =1$, the fourth Todd class $\td_4(X) $ is the class of a point $[\pt]$.
Thus
\[
\td(X) = 1+ \frac{1}{2}r_X H_X + \frac{1}{12}(r_X^2H_X^2+c_2(X))+\frac{1}{24}r_XH_Xc_2(X)+[\pt].
\]
Also,
\[
\td(T_\pi) = 1+ \frac{1}{2}(-K_\pi) + \frac{1}{12}(-K_\pi)^2-\frac{1}{720}(-K_\pi)^4.
\]
Thus
\begin{align*}
 \td(W) &= \pi^*\td(X)\cdot\td(T_\pi)\\
 &= \pi^*\left( 1+ \frac{1}{2}r_X H_X + \frac{1}{12}(r_X^2H_X^2+c_2(X))+\frac{1}{24}r_XH_Xc_2(X)+[\pt]\right)\\
 &\qquad\cdot\left(1+ \frac{1}{2}(-K_\pi) + \frac{1}{12}(-K_\pi)^2-\frac{1}{720}(-K_\pi)^4\right).
 \end{align*}
 Noting $-K_\pi = 2 \xi-r_X\pi^*H_X$ and the grothendieck relation \[
 \xi_\sE^2=\xi_\sE\cdot (r_X\pi^*H_X) - \pi^*c_2(\sE),
\]
we can calculate the right hand side of equation~\eqref{eq:RR}:
\begin{equation}
 \chi(m(a\xi_\sE-\pi^*H_X))  = f_1m^5+f_2m^4+f_3m^3+f_4m^2+f_5m+1,
 \end{equation}
 where
 
\begin{align*}
 f_1&= 
 \frac{a^5 r^4-5a^4 r^3+10a^3 r^2-10a^2 r+5a}{120}x
-\frac{3a^5 r^2-10a^4 r+10a^3}{120}y
+ \frac{a^5}{120}z,\\
%%%%%%%%%%%%%%%%%%%%%%%%%%%%%
 f_2&=
 \frac{a^4 r^4-4a^3r^3+6a^2r^2-4a r+1}{24}x
-\frac{3a^4r^2-8a^3r+6a^2}{24}y
 +\frac{a^4}{24}z,\\
 %%%%%%%%%%%%%%%%%%%%%%%%%%%%%%%
 f_3&=
 \frac{5 a^3 r^4-15 a^2 r^3+15 a r^2-6r}{72}x
 -\frac{5 a^3 r^2-10 a^2 r+4a}{24}y
 +\frac{a^3}{18}z\\
 &\qquad
 +\frac{a^3 r^2-3a^2 r+3a}{72}u
 -\frac{a^3}{72}v,&\\
 %%%%%%%%%%%%%%%%%%%%%%%%%%%%%%%
 f_4&= 
 \frac{a^2 r^4-2a r^3+r^2}{24}x
 -\frac{3a^2 r^2-4a r}{24}y
 +\frac{a^2 r^2-2a r+1}{24}u
 -\frac{a^2}{24}v,\\
 %%%%%%%%%%%%%%%%%%%%%%%%%%%%%%%
 f_5&=
 \frac{a r^4}{180}x
 -\frac{a r^2}{60}y
 -\frac{a}{45}z
 +\frac{2a r^2-3r}{72}u
 -\frac{a}{36}v+a.
\end{align*}

On the other hand, since $Y \simeq \bP^4$, we have
\[
\chi(m(a\xi_\sE-\pi^*H_X))=\chi(mH_Y)=\binom{m+4}{m}.
\]
Thus
\begin{align*}
 &f_1= 0,\\
 &f_2= \frac{1}{24},\\
 &f_3= \frac{5}{12},\\
 &f_4= \frac{35}{24},\\
 &f_5= \frac{25}{12}.
\end{align*}
Note that $(a,r_X) = (1,3)$ or $(3,1)$.
Thus, for each pair $(a,r_X)$, we can solve this system of linear equations for $(x,y,z,u,v)$.

If $a=3$ and $r_X=1$, then the above system of linear equations has no solution.
Thus we have $a=1$ and $r_X=3$.
Now the assertion follows from a straightforward calculation.
\end{proof}
\begin{proof}[Proof of Thoerem~\ref{thm:quarticDP}]
As in the proof of Theorem~\ref{thm:quarticDP}, we have the following commutative diagram:
\[
\xymatrix{
  \bP(\sE ) \ar[d]_-{\pi} \ar[r] \ar@(ur,ul)[rr]^-{\varphi}  & \bP(\sS ^* _{\Gr}) \ar[d]_-{\pi_{\Gr}} \ar[r] & \bP^4         \\
X \ar[r]^-{j} & \Gr(2,5),& \\
}
\]
with the condition $j^*\sS_{\Gr}^* \simeq \sE(-1)$.
In particular $j^*\cO_{\Gr}(1) \simeq \cO_X(1)$.
Thus the map $j$ is finite, and the image $V \coloneqq j(X)$ is a closed subvariety of dimension $4$.

Recall that the Chow group of $\Gr(2,5)$ is generated by the Schubert cycles: fix a complete flag $(0\subset V_1 \subset V_2\subset V_3\subset V_4\subset V_5 =\bC^5 )$ in the vector space $\bC^5$.
Then, for a pair $(a_1, a_2)\in \bZ_{\geq0}$ with $2 \geq a_1 \geq a_2$, the Schubert cycle $\sigma_{(a_1,a_2)}$ is defined as the cycle associated to the closed subset
\[
\Sigma_{(a_1,a_2)} \coloneqq \{\, W \in \Gr(2,5) \mid \dim (V_{3+i-a_i} \cap W)\geq i \,\}.
\]

With this notation, we have
\[
c(\sS^*_{\Gr}) =1+ \sigma _{1,0} +\sigma _{1,1}.
\]
Thus
\[
c(\sE) =1+ 3H_X +j^*(\sigma _{1,1})+2H_X^2.
\]

Since $V = j(X)$ is a closed subvariety of dimension $4$, the class $[V]$ is written as $c\sigma_{1,1} + d\sigma _{2,0}$.
Then $j_*(X)=m[V]= cm\sigma_{1,1} + dm\sigma _{2,0}$, where $m$ is the degree of the map $j$.
Then, we have
\begin{align*}
 x&= H_X^4 = m(3c+2d),\\
 y&= c_2(\sE)\cdot H_X^2 = m(7c+5d),\\
 z&= c_2(\sE)^2= m(17c+12d).
\end{align*}

On the other hand, by Proposition~\ref{prop:RR}, we have
$ y= \frac{5x-1}{2}$ and
$ z= \frac{13x-7}{2}$,
and hence we have $m=1$, $c=1$ and $d=1$.

Since $m=1$, the map $j$ is birational, and hence it is the normalization map of $V$.
Also, since $c=1$ and $d=1$, the subvariety $V$ is rationally equivalent to the codimension two linear section $V_5\subset \Gr(2,5)$ via the Pl\"ucker embedding $\Gr(2,5) \subset \bP^9$.
Moreover, since $H_X^4=5$, we already know that $X$ is a quintic del Pezzo $4$-fold.
Now it is easy to check that $V$ itself is a smooth quintic del Pezzo $4$-fold $V_5\subset \Gr(2,5)$ and the map $j$ is an isomorphism.
This completes the proof.
\end{proof}

\bibliographystyle{amsalpha}
\bibliography{}

\end{document}